\documentclass[12 pt, a4paper]{amsart}
\usepackage[T1]{fontenc}
\usepackage[latin9]{inputenc}
\usepackage{xcolor}
\usepackage{pdfcolmk}
\usepackage{amsthm}
\usepackage{amstext}
\usepackage{amssymb}
\usepackage{setspace} 

\PassOptionsToPackage{normalem}{ulem}
\usepackage{ulem}

\makeatletter

\providecolor{lyxadded}{rgb}{0,0,1}
\providecolor{lyxdeleted}{rgb}{1,0,0}

\usepackage{graphicx}

\usepackage{amsfonts}
\usepackage{amssymb,latexsym,epsfig,color}
\usepackage{enumerate}
\usepackage[all,cmtip]{xy}
\usepackage{cite}

\def \bigamalg {\coprod }

\newcommand{\N}{{\mathbb N}}
\newcommand{\R}{{\mathbb R}}
\newcommand{\Q}{{\mathbb Q}}

\newcommand{\bp}{{\textbf P}}
\newcommand{\bm}{{\textbf M}}

\newcommand{\OO}{{\mathcal O}}

\newcommand{\PP}{{\mathcal P}}

\newtheorem{thm}{Theorem}[section]
\newtheorem{theorem}{Theorem}[section]

\newtheorem{corollary}[theorem]{Corollary}
\newtheorem{lemma}[theorem]{Lemma}

\theoremstyle{definition}
\newtheorem{defn}[theorem]{Definition}
\newtheorem{example}[theorem]{Example}

\theoremstyle{remark}
\newtheorem{rem}[theorem]{Remark}
\begin{document}
\pagestyle{plain}
\title{Quantales, generalised premetrics and free locales}
\author{J. Bruno}
\author{P. Szeptycki}
\maketitle

\begin{abstract} Premetrics and premetrisable spaces have been long studied and their topological interrelationships are well-understood. Consider the category ${\bf Pre}$ of premetric spaces and $\epsilon$-$\delta$ continuous functions as morphisms. The absence of the triangle inequality implies that the faithful functor ${\bf Pre} \to {\bf Top}$ - where a premetric space is sent to the topological space it generates - is not full. Moreover, the sequential nature of topological spaces generated from objects in ${\bf Pre}$ indicates that this functor is not surjective on objects either. Developed from work by Flagg and Weiss, we illustrate an extension ${\bf Pre}\hookrightarrow  {\bf P} $ together with a faithful and surjective on objects left adjoint functor ${\bf P} \to {\bf Top}$ as an extension of ${\bf Pre} \to {\bf Top}$. We show this represents an {\it optimal} scenario given that ${\bf Pre} \to {\bf Top}$ preserves coproducts only. The objects in $\bp$ are metric-like objects valued on {\it value distributive lattices} whose limits and colimits we show to be generated by free locales on discrete sets.
  \end{abstract}

\section{Introduction}\label{sec:intro}

As a refinement of Kopperman's work from \cite{MR935419}, in \cite{MR1452402} Flagg introduces a family of metric-like objects with the property that any topological space can be naturally generated by one such object. More precisely, for a value quantale $V = (V,\leq, +)$ the author defines the notion of a \emph{V-continuity space} to be a pair $(X,d)$ where $X$ is a set and $d:X^2 \to V$ is a map for which $d(x,x) = 0$ and $d(x,y)\leq d(x,z)+d(y,z)$ for all $x,y,z \in X$. Flagg adopts Kopperman's terminology and denotes any triplet $(X,V,d)$ a \emph{continuity space} (where $V$ and $(X,d)$ are as defined previously). As a generalisation of metrisable spaces, Flagg illustrates how any continuity space naturally generates a topological space. Conversely, and perhaps surprisingly, for any topological space $(X,\tau)$ Flagg constructs a value quantale $\Omega(\tau)$ and a $\Omega(\tau)$-continuity space $(X,d)$ so that $(X,\tau)$ is generated by the continuity space $(X,\Omega(\tau),d)$. In particular, any metrisable topological space is generated by some $[0,\infty]$-continuity space. These ideas are further developed and grounded within a categorical setting by Weiss in \cite{MR3334228} by establishing an equivalence of categories $M:{\bf M} \leftrightarrows {\bf Top}:\mathcal{O}$ where the objects of {\bf M} are Flagg's continuity spaces. Morphisms in {\bf M} are extensions of $\epsilon$-$\delta$ continuous functions between metric spaces and are shown to be equivalent to continuous functions. This category {\bf M} is shown to be a natural extension of {\bf Met} - of all metric spaces - and the following diagram is established,
 \[
 \xymatrix{ 
  {\bf Top} \ar@/_/[rr]_M & &  {\bf M} \ar@/_/[ll]_{\mathcal{O}}  \\
  {\bf Top_M} \ar[u]  & &  {\bf Met}  \ar[u]\ar[ll]_{\mathcal{O}}\\
  }
\]
where ${\bf Top_M}$ is the category of metrisable topological spaces and arrows going up are inclusions. 

The category ${\bf Pre}$ is the one whose objects $(X,d)$ are premetric spaces. That is, $d: X^2 \to \R$ is a function where we only require $d(x,x)=0$ for all $x\in X$. Morphisms in ${\bf Pre}$ are $\epsilon$-$\delta$ continuous functions. There exists the obvious functor $\mathcal{O}:{\bf Pre} \to {\bf Top}_{\bf P}$ that extends $\mathcal{O}:{\bf Met} \to {\bf Top}_{\bf M}$, where a premetric space is sent to the topological space it generates and ${\bf Top}_{\bf P}$ is the category of premetrisable topological spaces; a subset $O$ of $X$ is $\tau_d$-open if, and only if, for any $x\in O$ we can find $\epsilon > 0$ so that $B_\epsilon(x) \subseteq O$. Here two important issues arise: (a) by the sequential nature of objects in ${\bf Pre}$ the functor is not surjective on objects\footnote{In general, any triplet $(X,V,d)$ where $V$ is a complete linear order will yield a {\it radial} topology. It is this fact that forces us to go beyond the realm of linearly ordered sets.} (b) since premetrics are not required to satisfy the triangle inequality, epsilon balls are not necessarily open in the generated topology - actually, the centre of an epsilon ball might not belong to its interior, even if it is not empty. Indeed, take $X = \{a,b,c,d\}$ with $d: X^2 \to \R$ as
 \[
 d(x,y) = 
 \begin{cases}
 0 & \text{ if } \{x,y\} = \{a,b\}\text{ or } \{x,y\} = \{b,c\},\\
 2 & \text{ if } \{x,y\} = \{a,c\}, \text{ and}\\
1 & \text{ otherwise}.
 \end{cases}
 \]

\noindent
The reader can quickly verify that int$[B_2(a)] = \{d\}$. Consequently, for (b) it also follows that $\epsilon$-$\delta$ continuous functions are topologically continuous but the converse is certainly not true. In other words, (b) says that $\mathcal{O}:{\bf Pre} \to {\bf Top}$ is not full and, thus, it is not possible to replicate the above equivalences with an extension of {\bf Pre}. In light of $M:{\bf M} \leftrightarrows {\bf Top}:\mathcal{O}$ it is natural to ask: {\it how much is lost by dropping the triangle inequality from {\bf M}?} 

Let {\bf P} be the category of \emph{generalised premetrics spaces} whose objects are triplets $(X,V,d)$ - where $X$ is a set, $V$ is a value distributive lattice and the map $d:X^2 \to V$ must satisfy $d(x,x) = 0$. Morphisms in {\bf P} are $\epsilon$-$\delta$ continuous functions like the ones in {\bf M}. In fact, we show later that ${\bf M}$ is a reflective subcategory of ${\bf P}$, thus highlighting a natural procedure for adding the triangle inequality to any generalised premetric space. In spite of (a) and (b), we show that {\bf P} is remarkably similar to ${\bf Top}$. More precisely, as an obvious extension of $\mathcal{O}:{\bf Pre} \to {\bf Top}$ we show that the functor $\mathcal{O}: {\bf P} \to {\bf Top}$ is left adjoint. In other words, for a large collection of categorical constructions in {\bf Top} it is only necessary to take into account $\mathcal{O}$-images of $\epsilon$-$\delta$ continuous functions. This is most unexpected given the large discrepancy between $\epsilon$-$\delta$ continuity and topological continuity; a fact we highlight in more detail in Section~\ref{sec:Primer} where we investigate several scenarios in which both types of continuity coincide.

The outline of the paper is the following. In Section 2 we investigate premetric spaces, topological continuity vs. $\epsilon$-$\delta$ continuity and briefly remark some interesting facts regarding various categories thereof. This section naturally leads to Section 3 where we illustrate {\bf P} as an extension of {\bf Pre} and explore topological continuity vs. $\epsilon$-$\delta$ continuity in its full generality. We close this section by proving that $\mathcal{O}: {\bf P} \to {\bf Top}$ is left adjoint. Section 4 is concerned with constructions of (co)limits in {\bf P}. We approach these constructions by proving that set-indexed (co)cones have $U$-initial(-final) lifts where $U:{\bf P} \to {\bf Set}$ is the usual forgetful functor.

  \section{A primer on premetric spaces and sequential spaces}\label{sec:Primer} We illustrate some basic facts regarding premetrisability and topological continuity vs $\epsilon$-$\delta$ continuity. The topology $\tau_d$ on a set $X$ generated by a premetric $d:X^2\to \R$ is the one for which $U\in \tau_d$ if, and only if, for all $x\in U$ there exists an $\epsilon>0$ so that $U\supseteq B_\epsilon(x)=\{y\mid d(x,y)<\epsilon\}$ (it is a relatively straightforward task to show that any such topology is sequential\cite{MR1077251}). It is important to notice that, in general, $B_\epsilon(x)$ might not be open; the interior of such a set might not even contain $x$ itself. Hence, sequential convergence is not the same for $d$ as it is for $\tau_d$; it is only possible to claim that $d$-convergence implies $\tau_d$-convergence. Consequently, the reader can quickly verify that $\epsilon$-$\delta$ continuity always implies topological continuity between any pair of premetric spaces; the converse is not true. As a matter of fact, the equivalence between both types of continuity occurs precisely when the same is true for both types of convergence. Recall that any function $f:(X,\rho) \to (Y,\sigma)$ between sequential spaces is continuous if, and only if, for any sequence $(x_n)$ in $X$ and $x\in X$ we have $((x_n)\rightarrow_\rho x \Rightarrow f(x_n)\rightarrow_\sigma f(x))$. It is not hard to verify that the very same holds for premetric spaces.
  
  \begin{lemma} A function $f:(X,d)\to(Y,m)$ between premetric spaces is $\epsilon$-$\delta$ continuous if, and only if, for any sequence $(x_n)$ in $X$ and $x\in X$ we have 
  
  \[(x_n) \rightarrow_d x \Rightarrow f(x_n) \rightarrow_m f(x).\]
  \end{lemma}
   In fact the coincidence of both types of continuity depends only on the codomain of the function.
  \begin{lemma} For any premetric space $(Y,m)$ the following are equivalent.
  \medskip
  \begin{itemize}
  \item For any function $(X,d) \to (Y,m)$, topological continuity and $\epsilon$-$\delta$ continuity coincide. 
  \medskip
  \item The notions of $\tau_m$-convergence and $m$-convergence coincide.
  \end{itemize} 
  \end{lemma}
  \begin{proof} ($\Leftarrow$) Whether or not $\tau_m$-convergence implies $m$-convergence, $\epsilon$-$\delta$ continuity implies topological continuity. Hence, assume that $\tau_m$-convergence implies $m$-convergence and let $f: (X,d) \to (Y,m)$ be any function which is not $\epsilon$-$\delta$ continuous. This means that there exists an $\epsilon >0$ and an $x\in X$ for which given any $\delta >0$ we can find $y\in X$ with $d(x,y)<\delta$ but $m(f(x),f(y))\geq \epsilon$. Consequently, there exists a $d$-convergent (and thus $\tau_d$-convergent) sequence $(y_n) \rightarrow x$ for which $f(y_n) \not \rightarrow f(x)$ with respect to $m$. Since we assumed that $\tau_m$-convergence implies $m$-convergence, then $f(y_n) \not \rightarrow f(x)$ with respect to $\tau_m$ either and $f$ is not continuous.\\
  $(\Rightarrow)$ Assume $\tau_m$-convergence to be strictly weaker than
$m$-convergence and let convergent sequence $(x_n)\to x$ in $Y$ be a witness of this fact. Take the convergent sequence space $(\omega+1,d)$ with, say, $d(n,\omega) = \frac{1}{n}$ and the map $ \omega \to Y$ for which $n \mapsto x_n$ and $\omega \mapsto x$. This map is continuous but not $\epsilon$-$\delta$ continuous.
\end{proof}  
  
  \begin{corollary}\label{cor:equiv} For any function $(X,d) \to (Y,m)$ between premetric spaces, each of the following conditions imply that both types of continuity coincide. 
\medskip
  \begin{enumerate}
  \item For all $y\in Y$ and $\epsilon > 0$, $B_\epsilon^m(y)$ is open with respect to $\tau_m$.
  \medskip
  \item $m$ satisfies the triangle inequality.
    \medskip
  \item $\tau_m$-sequential limits are unique.
   \medskip
  \item $\tau_m$ is $T_2$.
  \end{enumerate}
  \end{corollary}
  \begin{proof} The proof of (1) follows from standard arguments for metric spaces and (2) follows from (1). For (4) notice that $T_2$ implies uniqueness of limits, therefore we focus on proving (3). For the remaining case we show that $\tau_m$-convergence implies $m$-convergence. Assume that for some sequence $(x_n)$ in $Y$ we have that $(x_n) \rightarrow_{\tau_m} x$ but $(x_n) \not \rightarrow_m x$. Let $\epsilon >0$ so that $B_\epsilon(x) \cap \{x_n\mid n\in M\} =\emptyset$ for some infinite $M\subseteq \N$. By uniqueness of limits we can deduce that for any $y \in B_\epsilon(x)\smallsetminus \{x\} := B$, $y$ cannot be a limit point of the set $\{x_n\mid n\in M\}$. If that was the case, then one could find a subsequence of $(x_n)_{n\in M}$ converging to that $y$ and thus a contradiction. Hence, for each $y \in B$ let $U_y$ be any open set containing $y$ and so that $U_y \cap \{x_n\mid n\in M\} =\emptyset$. It follows that 
  \[
  \displaystyle \bigcup_{y \in B} U_y \cup B_\epsilon(x)
  \]
 is open, contains $x$ and is disjoint from $(x_n)_{n\in M}$. Hence, $(x_n)$ does not $\tau_m$-converge to $x$ either and we arrive at a contradiction. Consequently, by the previous lemma both types of continuity agree.\end{proof}
  
 None of the above conditions are necessary. The previous lemma narrows the scope of candidates for a topologically continuous function $(X,d) \to (Y,m)$ which is not $\epsilon$-$\delta$ continuous: $(Y,\tau_m)$ must be at most $T_1$ and $(Y,m)$ must not have unique limits of sequences. The following example illustrates just that.
  
 \begin{example} 
Take two countably infinite disjoint sets $A=\{a_i\mid i\geq 2\},B=\{b_i\mid i\geq 2\}$ and let $Y=A\cup B$. Define $m:Y^2 \to \R$ by 
 \[
 m(x,y) = 
 \begin{cases}
 \frac{1}{\text{max}\{n,m\}} & \text{ if } \{x,y\} = \{a_n,b_m\}\text{ for a pair }n,m\in \N\\
1 & \text{ otherwise}.
 \end{cases}
 \]
 \noindent
 Notice that $m$ is {\it symmetric} (i.e., $m(x,y) = m(y,x)$) and {\it separated} (i.e., $m(x,y)>0$ for $x\not = y$), hence, $d$ fails only to satisfy the triangle inequality. Generate a topology $\tau$ on $Y$ as usual: $O\in \tau_m$ if, and only if, for any $x\in O$ there exists an $\epsilon>0$ so that $B_\epsilon(x) \subseteq O$. By design, $A$ is the set of limit points of $B$ and $B$ is the set of limit points of $A$. Also, it is simple to observe that any open set is cofinite. The converse is also true. Let $O$ be cofinite and $p,q \in \N$ so that $p$ is the least number for which $\forall i\geq p$, $a_i \in O$ and $q$ is the least number for which $\forall i\geq q$, $b_i \in O$. For each $a_i \in O$ let $\delta = \frac{1}{\text{max}\{q,i\}}$ and notice that $B_\delta(a_i)\subset O$. The same is true for all $b_i \in O$ and hence, $O$ is open. 
Next, split the rationals into two mutually dense sets $C,D$ and let $f:\Q\to Y$ be a bijection for which $f(C) = A$, $f(D) = B$. Assume $d$ is the usual metric on $\Q$ and notice that since $\tau_m$ is the cofinite topology on $Y$, then $f$ is topologically continuous. However, $f$ fails to be $\epsilon$-$\delta$ continuous about each and every point in its domain. Indeed, without loss of generality, let $x\in C$ and notice that for any $\delta> 0 $, $B_\delta(x)$ contains infinitely many points from $D$. Whence, choosing $\epsilon = \frac{1}{2}$ yields that for all $\delta > 0$ $f(B_\delta(x)) \not \subseteq B_\epsilon(f(x))$ and that $f$ is not $\epsilon$-$\delta$ continuous.

 \end{example}

Let {\bf Pre}, ${\bf Top}_{\bf P}$ and {\bf Seq} denote the categories of premetric spaces (with $\epsilon$-$\delta$ continuous), premetrisable topologies and sequential topologies (with continuous functions as morphisms in the latter two), respectively. It is known that the bicomplete category {\bf Seq} is a coreflective subcategory of {\bf Top}. It is not closed under topological limits; limits in {\bf Seq} are constructed by applying the convergent-open topology to underlying products and subsets (see \cite{MR0180954} and \cite{MR0222832}). Not all sequential spaces are premetrisable. In fact, more is true: neither Fr\'{e}chet nor premetrisability imply each other. Example 5.1 from \cite{MR0222832} illustrates a premetrisable space that is not Fr\'{e}chet. For a Fr\'{e}chet space that is not premetrisable consider the following example.

\begin{example}\label{exam:frechetfan} Let $X$ be the topological products of countably many copies of $\omega + 1$ and quotient all of $\omega \times \{\omega\}$, and denote this point $\infty$ (this is the {\it Fr\'{e}chet Fan}). The resulting space is Fr\'{e}chet but fails to be premetrisable. Indeed, one can easily verify that epsilon balls must be open in $X$ for any premetric on it since $X\smallsetminus \infty$ is discrete. However, the Fr\'{e}chet Fan is not first countable and, thus, can not be premetrised.
\end{example}

The Fr\'{e}chet Fan is an excellent example of a space that cannot be generated by means of evaluating distances between points on a linear order. Both conditions are weaker than first-countability (that Fr\'{e}chet is weaker is a well-known result). 

\begin{lemma} Any first countable space is premetrisable.
\end{lemma}

\begin{proof}Take any first countable space $X$, and for each point $x\in X$ select a countable nested collection of neighbourhoods about it, $U_n(x)$. Let $d:X^2 \to \R$ as

 \[
 d(x,y) = 
 \begin{cases}
 0 & \text{ if } y \in \bigcap_n U_n(x),\\
 \frac{1}{n} & \text{ if } y\in U_n(x)\smallsetminus U_{n+1}(x),\\
1 & \text{ otherwise}.
 \end{cases}
 \]
\end{proof}

The category {\bf Pre} has equalisers and coproducts. The former are simple to construct: for a premetric space $(X,d)$ and $E\subseteq X$, take $m$ on $E$ as the restriction of $d$ on $E$. The functor ${\bf Pre} \to {\bf Top}$ does not preserve equalisers. Example 5.1 in \cite{MR0222832} illustrates a sequential space with a non-sequential subspace and since ${\bf Pre}$ is closed under equalisers the claim follows. As for coproducts, let $\phi:\R \to [0,1)$ be any order embedding and for a collection $(X_i,d_i)_{i\in I}$ define $Y = \prod_i X$, and $m: Y^2 \to \R$ so that $m(x,y) = {\phi}\circ d_i(x,y)$ when $x,y\in X_i$ (for some $i$), and $1$ otherwise. A moment's thought verifies that ${\bf Pre} \to {\bf Top}$ preserves these colimits and that they also exist in ${\bf Top}_{\bf P}$. We use Example~\ref{exam:frechetfan} to show that coequalisers do not exist in {\bf Pre}: let $(X,d)$ be the sum in {\bf Pre} of $\omega$ copies of the convergent sequence $(\frac{1}{n})\cup\{0\}$ with the usual metric. Let $Y$ be the quotient set where by all limit points (i.e., the 0's) are glued together, and assume that for some premetric $r:Y^2\to \R$, $(Y,r)$ represents the coequaliser of the above scenario in {\bf Pre}. Since the quotient map $q:X \to Y$ must be $\epsilon$-$\delta$ continuous it follows that any epsilon ball about $q(0) = \overline{0}$ contains all but finitely many elements of each convergent sequence in $X$. Next, notice that for any open set in the Fr\'{e}chet Fan, one can forge a metric on $Y$, say $s$, that would witness such an open set making the function $q:(X,d) \to (Y,s)$ $\epsilon$-$\delta$ continuous. Since $(Y,r)$ was assumed to be the coequaliser, then $id_Y: (Y,r) \to (Y,s)$ must also be $\epsilon$-$\delta$ continuous and, in turn, $(Y,r)$ would generate the Fr\'{e}chet Fan on $Y$. A contradiction.

 \section{Constructing {\bf P} and $\mathcal{O}:{\bf P} \to {\bf Top}$}\label{sec:const}
 We begin by illustrating the construction of {\bf P} and later describe its relationship with {\bf Top}. First, we begin by recalling some basic facts and definitions regarding lattices (see \cite{MR1452402}). For a lattice $L$ and a pair $x,y \in L$, we say that $y$ is \emph{well above} $x$ and write $y\succ x$ if whenever $x\ge\bigwedge S$, with $S\subseteq L$, there exists some $s\in S$ such that $y\ge s$. A well-known characterization of completely distributive lattices (see \cite{MR2612143}) is the following.

\begin{theorem} A lattice $L$ is completely distributive if, and only if, for all $y \in L$ 
$$
y = \bigwedge \{a\in L\mid a \succ y\}.
$$
\end{theorem}
 
A \emph{value distributive lattice} is a completely distributive lattice $V$ for which $V_\prec = \{a \in V \mid a \succ 0\}$ forms a filter. A simple example of a value distributive lattice is the extended positive real line $[0,\infty]$. Following previous work from Flagg and Kopperman we define the following.



\medskip

\begin{defn} Given a value distributive lattice $V$, a \emph{$V$-space} is a pair $(X,d)$ so that $d:X\times X \to V$ for which $d(x,x) = 0$ for all $x\in X$. Any triple $(V,X,d)$, where $(X,d)$ is a $V$-space is called a \emph{continuity space}.
\end{defn}

 The category ${\bf P}$ will be that of all continuity spaces. Objects in {\bf P} are the triples
$(V,X,d)$ where $V$ is a value distributive lattice and $(X,d)$ is a $V$-space, and a morphism $(V,X,d)\to(W,Y,m)$ is a function $f:X\to Y$ such that for
every $x\in X$ and for every $\epsilon\in W_\prec$
there exists $\delta\in V_\prec$ such that for all
$x'\in X$ if $d(x',x)\prec \delta$ then $m(f(x'),f(x)) \prec \epsilon$. We will refer to these morphisms as \emph{ $\epsilon$-$\delta$ continuous functions}. Every ordinary premetric
space $(X,d)$ is a $V$-space for $V=[0,\infty]$. Hence, ${\bf Pre}$ is a full subcategory of ${\bf P}$.

\begin{rem} The term $V$-space, for a value quantale $V = (V,+)$, was initially introduced in \cite{MR1452402} to denote a pair $(X,d)$ so that $d:X\times X \to V$ for which $d(x,x) = 0$ for all $x\in X$ and $d(x,y)\leq d(x,z)+d(y,z)$ for all $x,y,z \in X$. In this paper we do not require transitivity when referring to $V$-spaces and thus we have no need for a binary operation $+:V^2 \to V$. Nonetheless, in the sequel we illustrate a natural way to introduce the triangle inequality to objects in ${\bf P}$ based on the equivalence ${\bf M} \leftrightarrows {\bf Top}$.
 \end{rem}
 
\begin{defn}
Let $(X,d)$ be a $V$-space and $\epsilon\in V$ with $\epsilon\succ0$. The set \emph{
}$B_{\epsilon}(x)=\{y\in X\mid d(y,x)\prec\epsilon\}$ is the \emph{$\epsilon$-ball }with radius $\epsilon$ about the point $x\in X$.
 \end{defn}

 \begin{lemma}
Let $(X,d)$ be a $V$-space. Declaring a set $U\subseteq X$ to be open if for every $x\in U$ there exists $\epsilon\succ0$ such
that $B_{\epsilon}(x)\subseteq U$ defines a topology on $X$. 
\end{lemma}

\begin{proof}
For a continuity space $(X,V,d)$ let $\tau$ be the collection of all $U\subset X$ for which the hypothesis is satisfied. Clearly, $\tau$ is closed under unions. The rest follows from the well-above relation. That is, let $U_1,U_2\in \tau$ and $x\in U_1 \cap U_2$. By definition, we can find $\epsilon_1,\epsilon_2\in V_\prec$ so that $B_{\epsilon_1}(x) \subseteq U_1$ and $B_{\epsilon_2}(x) \subseteq U_2$. Since $V$ is a value distributive lattice, then $\delta = \epsilon_1 \wedge \epsilon_2 \in V_\prec$ and $B_\delta(x) \subseteq B_{\epsilon_1}(x) \cap B_{\epsilon_2}(x) \subseteq U_1 \cap U_2$.
\end{proof}

For any collection of sets $X$ and $A\subseteq X$, we say that $A$ is \emph{downwards closed} provided that $B,C \in X$ and $B\subseteq C$, and $C\in A$ then $B\in A$. Also, we follow standard set-theoretic notation in that for any set $X$, we let $[X]^{<\omega}$ denote the collection of all finite subsets of $X$. The following construction is key for developing (co)limits in $\bp$. For any set $X$ let 

\[\Omega(X)=\{ A\subseteq [X]^{< \omega} \mid A \text{ is downwards closed}\}.\]

\begin{lemma} Given a set $X$, ordering $\Omega(X)$ by reverse set inclusion yields $(\Omega(X), \leq)$ as a value distributive lattice where $p \succ 0$ if, and only if, $p$ is finite.
\end{lemma}
\begin{proof} This is part of Example 1.1 in \cite{MR1452402}.
\end{proof}

\begin{theorem}\label{thm:PtoTop}
The functor $\mathcal{O}:{\bf P}\to{\bf Top}$ which sends a continuity space to the topological space it generates is surjective on objects and faithful. 
\end{theorem}

\begin{proof}
Verifying that $\epsilon$-$\delta$ continuous functions are also continuous is done in very much the same way as with premetric spaces. The following is due to Flagg and can be found in \cite{MR1452402} pg. 273: to show surjectivity of $\mathcal{O}$ take any topological space $(X,\tau)$ and construct an $\Omega(\tau)$-space
$(X,d)$ for which 
$$d(x,y)=\{F\in[\tau]^{<\omega}\mid\mbox{for all } U \in F \mbox{ if }  x \in U \mbox{ then } y \in U \}.$$
\noindent
Let $x\in U \in \tau$ and denote $\epsilon = \{\emptyset , \{U\}\}$. Construct $B_\epsilon(x)$ and notice
$$y\in B_\epsilon(x) \Rightarrow d(x,y) \prec \epsilon \Rightarrow  d(x,y) \supseteq \epsilon \Rightarrow y\in U.$$
\end{proof}

In Section 3.2 we show that $\mathcal{O}:{\bf P}\to{\bf Top}$ is left adjoint. The construction $\Omega(X)$ for a set $X$ is the dual of the free locale on $X$ (see \cite{MR861951}). We will frequently employ this construction in the sequel when developing (co)limits in {\bf P}. 

\subsection{Topological continuity vs $\epsilon$-$\delta$ continuity in $\bp$} In much the same spirit as with premetric spaces we show that when topological net convergence implies $\epsilon$-$\delta$ net convergence, topological continuity is equivalent to $\epsilon$-$\delta$ continuity. 

\begin{defn} Let $(X,V,d)$ be any continuity space and $(x_i)_{i\in I}$ be any net in $X$. We say that $(x_n)$ \emph{$d$-converges to a point $x\in X$} whenever for all $\epsilon \succ 0$ there exists $i_0 \in I$ so that for all $i\geq i_0$, $x_i \in B_\epsilon(x)$.
\end{defn}

By definition, $d$-convergence is stronger than $\tau_d$-convergence. Recall that topological continuity can also be characterised in terms of nets: a function $f:X \to Y$ is continuous if, and only if, it preserves net convergence. The following is then straightforward to prove.

  \begin{lemma} A function $f:(X,V,d)\to(Y,W,m)$ between continuity spaces is $\epsilon$-$\delta$ continuous if, and only if, $f$ preserves net convergence.
  \end{lemma}
  \begin{lemma} For any continuity space $(Y,W,m)$ the following are equivalent.
  \smallskip
  \begin{enumerate}
  \item For any function $(X,V,d) \to (Y,W,m)$, topological continuity and $\epsilon$-$\delta$ continuity coincide. 
  \medskip
  \item The notions of $\tau_m$-convergence and $m$-convergence coincide.\\
  \end{enumerate} 
  \end{lemma}

Mimicking the behaviour of premetric spaces and applying the arguments used in Corollary~\ref{cor:equiv} one can easily establish the following conditions on generalised premetric spaces.

   \begin{corollary}\label{cor:coincont} For any function $(X,V,d) \to (Y,W,m)$ between continuity spaces the following conditions imply that both types of continuity coincide. 
\medskip
  \begin{enumerate}
    \item For all $y\in Y$ and $\epsilon \succ 0_W$, $B_\epsilon^m(y)$ is open with respect to $\tau_m$.
    \medskip
  \item $\tau_m$-net limits are unique.
   \medskip
  \item $\tau_m$ is $T_2$.
  \end{enumerate}
  \end{corollary}

As a simple consequence of Corollary~\ref{cor:coincont} we have the following.

\begin{corollary} The functor $\mathcal{O}:{\bf P}\to{\bf Top}$ is a left adjoint.
\end{corollary}
\begin{proof} Given any topological space $(X,\tau)$, Theorem~\ref{thm:PtoTop} generates a {\bf P}-object $X_0=(X,\Omega(\tau),d)$ so that $\mathcal{O}(X_0) = (X,\tau)$. Moreover, one can easily verify that for all $x\in X$ and $\epsilon \succ 0$, $B_\epsilon(x) \in \tau$. Next, consider a continuity space $Y_0 = (Y,W,m)$ with a topologically continuous function $f: \mathcal{O}(Y_0) \to (X,\tau)$. From part (1) of the previous corollary we obtain that $f$ is also $\epsilon$-$\delta$ continuous and thus $\mathcal{O}$ is left adjoint.
\end{proof}

In view of the above result and the equivalence ${\bf M} \leftrightarrows {\bf Top}$ highlighted in Section~\ref{sec:intro} (see \cite{MR3334228}), it is clear that {\bf M} becomes a reflective subcategory of {\bf P}. This highlights a cohesive way in which to add the triangle inequality to any object in {\bf P}.

\begin{corollary} The category {\bf M} is a reflective subcategory of {\bf P}.
\end{corollary}

\section{Bicompletenness of {\bf P}}\label{sec:cont}

 Let $U:{\bf P}\to{\bf Set}$ denote the forgetful functor: we start this section by proving that any set-indexed cone has a $U$-initial lift. In a sense, $U$-initial lifts are analogous to initial topologies. In light of this, $\epsilon$-$\delta$ continuity being stronger than topological continuity can be interpreted as `topological continuity imposes more restrictions on initial lifts in {\bf Top} than $\epsilon$-$\delta$ continuity does in {\bf P}' (i.e., there are {\it more} of the latter than the former). Inevitably, the functor $\mathcal{O}$ in general does not preserve limits. In fact, it is simple to show that limits in {\bf P} are mapped to topologies at least as fine as their corresponding limits in {\bf Top}. In contrast, recall that the coincidence of both types of continuity is solely due to the codomain of a function. This fact lies at the very heart of why $\mathcal{O}$ is left adjoint and thus preserves colimits.

\subsection{Limits}
 In order to lighten the notational burden, in what follows we will suppress the subscript dummy indexing in the product notation. For instance, $\prod_{j \in J} V_j$ will become  $\prod V_j$ (where the indexing set will be understood from context). Also, when faced with a collection $\{V_j\}_{j\in J}$ of value distributive lattices, their bottom and top elements will be denoted by $\bot_j$ and $\top_j$, respectively - thus suppressing the `$V$' from their indices. 

\begin{theorem}\label{thm:topological} Any set-indexed cone $(f_j:X\to U[(X_{j}, V_j,d_j)])_{j\in J}$ has a $U$-initial lift.
\end{theorem}
\begin{proof} Fix any such cone $(f_j:X\to U[(X_{j}, V_j,d_j)])_{j\in J}$ and notice that for each $j \in J$ one can construct $(X, V_j, m_j)$ where $m_j(x,y) = d_j(f_j(x),f_j(y))$ and, thus, endow each $f_j$ with $\epsilon$-$\delta$ continuity. Next we construct a value distributive lattice $V$ (based on all $V_j$'s) and $m:X^2 \to V$ (based on all $d_j$'s) making each $f_j$ $\epsilon$-$\delta$ continuous (in addition to the usual cohesion properties of initial lifts). First notice that letting $V = \prod(V_{j})$ and $m(x,y) \in V$ so that $\pi_j\circ m(x,y) = m_j(x,y)$ does turn each function $f_j$ into an $\epsilon$-$\delta$ continuous function. However, $\prod(V_{j})$ is not
value distributive (the well-above elements do not form a filter). Indeed, take a product $L \times L$ of a value distributive lattice $L$. Its well-above zero elements are of the form $(\top_L,a)$ and $(b,\top_L)$ where $a, b\succ \bot_L$. Thus, their meet $(\top_L,a) \wedge (b,\top_L) = (b,a)$ is not well-above zero in $L\times L$ unless either $a$ or $b = \top_L$ . In order to fix this, we order-embed $\prod V_j$ into a suitable value distributive lattice $V$ and define $m:X^2 \to V$ accordingly. Recall that for any lattice $L$, the set $L_\prec:=\{a\in L \mid a\succ 0\}$. Let $U:=\prod_{f}(V_{j})_{\prec}$; $a\in U$ implies that for only finitely many $j \in J$, $a_j\not=\top_j$. Put $V := \Omega(U)$. The injection $\phi: \prod V_j \to V$ is defined as follows: for a given $ x\in\prod V_j$
let 
\[
\phi(x) = x_{\uparrow}=\{A\in[U]^{<\omega}\mid A\subseteq x^{\uparrow}\}
\]

\noindent
and $x^{\uparrow}=\{a\in U\mid \forall j\in J, a_{j}\succ x_{j}\}$. 
Notice that since all $V_{j}$ are completely distributive lattices
then for any $x\in\prod V_{j}$, $x^{\uparrow}$ uniquely
determines $x$. Consequently, we have $\bigwedge(\cup x_{\uparrow})=\bigwedge(x^{\uparrow})=x$
in $\prod V_{j}$.\\

\smallskip

\noindent
CLAIM:
the function $ \phi:\prod V_{j}\rightarrow V$
where $x\mapsto x_{\uparrow}$ is an order-embedding.

\begin{proof}
Take $x=(x_{j})$ and $y=(y_{j})$ in $\prod V_{j}$ so that
$x\not=y$. Notice that 
\[\bigwedge(\cup x_{\uparrow})=\bigwedge{\{a\in\cup x_{\uparrow}\mid a_{i}\geq x_{i}\}}=x\not=y=\bigwedge{\{a\in\cup y_{\uparrow}\mid a_{i}\geq y_{i}\}}=\bigwedge(\cup y_{\uparrow})\]
and, hence, that $\phi$ is injective. Also, if $x>y$ then clearly
$x_{\uparrow}\subset y_{\uparrow}$ and $x_{\uparrow}>y_{\uparrow}$.
\end{proof}

\noindent
Next, we define $m:X^2\rightarrow V$ as follows: for $x, y \in X$ let $d(x,y) \in \prod V_j$ so that 
\[\pi_j(d(x,y)) = d_j(x,y)
\]
 and
\[
m(x,y)=\phi\circ d(x,y).\\
\]

\smallskip

\noindent
CLAIM: for each $j\in J$, $f_j: (X,V,m) \to (X,V_j,d_j)$ is $\epsilon$-$\delta$ continuous. 
\begin{proof} Fix an $i \in J$ and let $\epsilon \succ \bot_i$. Let $\hat{\epsilon} \in  \prod_{f}(V_{j})_{\prec}$ so that $\pi_j(\hat{\epsilon}) = \top_j$ when $j\not = i$ and $\pi_j(\hat{\epsilon}) = \epsilon$, otherwise. Notice that $\overline{\epsilon} = \{\emptyset, \{\hat{\epsilon}\}\}\succ \bot_V$ and that $m(x,y) \prec \overline{\epsilon} \Rightarrow d_i(x,y) \prec \epsilon$. Thus, $f_i$ is as claimed.
\end{proof}

Lastly, choose any continuity space $(Z,W,s)$ and function $f:Z\to X$, and assume that all compositions $h_j:=f_j\circ f$ are $\epsilon$-$\delta$ continuous. Choose any $z\in Z$ and $\epsilon \succ \bot_V$ and recall that the latter is equivalent to $|\epsilon|\in \omega$. In particular, $|\cup \epsilon| \in \omega$ also. By construction, for each $p\in \cup \epsilon$ only finitely many projections onto their respective value distributive lattices are  different from the largest element of the given lattice at the given coordinate. Let 
$$k = \{t\mid \exists p \in \cup \epsilon \text{ and }j\in J, \pi_j(p)=t < \top_j\}.$$ 
\noindent
For the chosen $z$ and each $t\in k$ there exists a $\delta_t$ so that $s(x,y) \prec \delta_t \Rightarrow d_j(h_j(x),h_j(y)) \prec t$ (since each $f_j$ is $\epsilon$-$\delta$ continuous), where $t = \pi_j(p)$ for some $j\in J$ and $p \in \cup \epsilon$. Since $W$ is a value distributive lattice, then $\delta: = \bigwedge_{t \in k} \delta_t \succ 0_W$ and one can easily verify that $s(x,y) \prec\delta \Rightarrow m(f(z),f(y)) \prec \epsilon$.
\end{proof}

Forging products and equalisers in {\bf P} follows directly from the construction illustrated in Theorem~\ref{thm:topological}. In sum, we have the following.

\subsubsection{Products}
Take an arbitrary collection of continuity spaces $\{(X_{j}, V_j,d_j)\}$,
where $j\in J$, let $X=\prod X_{j}$ and $U=\prod_{f}(V_{j})_{\prec}$,
so that if $a\in U$ then for only finitely many $i \in J$, $a_i\not=1_i$ and let $V= \Omega(U)$. The injection $\phi: \prod V_j \to V$ is defined as follows: for a given $ x\in\prod V_j$
let 
$$
\phi(x) = x_{\uparrow}=\{A\in[U]^{<\omega}\mid A\subseteq x^{\uparrow}\}
$$
and $x^{\uparrow}=\{a\in U\mid \forall j\in J, a_{j}\succ x_{j}\}$. For $x=(x_{j}),y=(y_{j})\in X$ let $d(x,y) \in \prod V_j$ so that $\pi_j(d(x,y)) = d_j(x_j,y_j)$ and
\[
m(x,y)=\phi\circ d(x,y).
\]
\noindent
\begin{corollary} For a set-indexed collection of continuity spaces $\{(X_{j}, V_j,d_j)\mid j\in J\}$ we have $\prod_{\bf P}  (X_j,V_j,d_j) = ((X,V,m), X \to X_i)$ with $X$, $V$ and $m$ as illustrated above.
\end{corollary}

\subsubsection{Equalisers} These represent the simplest of all constructions. Take a pair of continuity spaces $(V,X, d_X)$ and $(Y,W,d_Y)$ with
$\epsilon$-$\delta$ continuous $f,g : (X,V,d_X) \rightrightarrows (Y,W,d_Y)$. The
equaliser is simply $Z=\{x\in X\mid f(x)=g(x)\}$ with $V$ and $d_{Z}:X\times X\rightarrow V$ as the restriction of $d_X$ onto $Z$. The inclusion function $Z \to X$ is clearly $\epsilon$-$\delta$ continuous.

\subsection{Colimits} 

In parallel with limits, we first prove that any set indexed cocone has a $U$-final lift and concurrently expose coproducts and coequalisers as part of the proof. The construction of these colimits in $\bp$ requires of more delicate arguments than with limits. This is particularly true for coequalisers.

\begin{theorem}\label{thm:cotopological} Any set-indexed cocone $( f_j:U[(X_{j}, V_j,d_j)] \to X)_{j\in J}$ has a $U$-final lift.
\end{theorem}

\begin{proof} In much the same spirit as with Theorem~\ref{thm:topological}, this proof is naturally broken up into two major parts. We first fix a set-indexed $( f_j:U[(X_{j}, V_j,d_j)] \to X)_{j\in J}$ and for any $j\in J$ we construct a continuity space $(X,W_j,m_j)$ rendering $f_j$ $\epsilon$-$\delta$ continuous. En route to achieving this we develop coequalisers. Secondly, based on the collection $\{(X,W_j,m_j)\mid j\in J\}$ we then construct the $U$-final lift $(X,W,m)$ and in passing we develop coproducts.

Consider a function $f:U[(Y,V,d)]\to X$: in what follows we construct the equivalent of the final topology on $X$ as a continuity space $(X,W,m)$ where the value distributive lattice $W$ is based on 

\[
\left[(V)_\prec\right]^{Y} = \{ \text{all functions }h: Y \to (V)_\prec\}
\]

\medskip
\noindent
 and a distance assignment $m: X^2 \to W$ based on finite paths in $Y$ (i.e., finite sequences $x_1, \ldots, x_n$ in $Y$) and show that $f: (Y, V,d)\to(X,W,m)$ is $\epsilon$-$\delta$ continuous. For further reference and in line with topology, let us denote this construction as the \emph{final continuity space} for a function $f:U[(Y,V,d)]\to X$. Put $M = \left[(V)_\prec\right]^{Y}$ and let $W = \Omega\left(M\right)$.

For $A \in\left[M\right]^{<\omega}$ and a pair $a,b \in f(X)$ say that $A$ \emph{admits} $(a,b)$ if for all $h\in A$ there exists a finite sequence $x_1, \ldots, x_n$ in $Y$ so that:
\begin{itemize}
\item $f(a) = f(x_1)$ and $f(b) = f(x_n) $,
\item for $i$ odd, $x_{i+1}\in B^{d}_{h(x_i)}(x_i)$, and 
\item for $i$ even $f(x_{i})= f( x_{i+1})$.
\end{itemize}

\medskip
\noindent
 Next, define $m: X^2 \to W$ as follows: for all $x,y\in X$ let

\[
m(x,y) = \begin{cases}
\{ A \in\left[M\right]^{<\omega}\mid A \text{ admits } (x,y)\} & \text{if } x,y\in f(Y),\\
\top_{W} & \mbox{otherwise.}
\end{cases}
\]
Clearly, for all $x,y \in X$ we have  $m(x,y) \in W$ and points in $X\smallsetminus f(Y)$ are as far from all other points (and vice versa) as possible and generate the discrete topology on $X\smallsetminus f(Y)$. Let $\epsilon \succ \bot_{W}$ and let $x \in Y$. Define
\[
\delta = \bigwedge_{h\in \cup \epsilon} h(x)
\]
and notice that since $\cup \epsilon$ is finite and each $h(x) \succ \bot_{W}$ then $\delta \succ \bot_{W}$. By design, $f\left[B^{d}_\delta(x)\right] \subseteq B_\epsilon^{m}(f(x))$ and notice that the inverse image of the latter is a saturated open set. Moreover, the reader can easily verify that given any function $g:X\to U(Z_0)$ so that $g\circ f: (Y,V,d) \to Z_0$ is $\epsilon$-$\delta$ continuous, then it must be that $g: (X,W,m) \to Z_0$ is also. Consequently, when $X = f(Y)$ the final continuity space on $X$ is equivalent to the coequaliser on $Y$ with the equivalence on $Y$ by $x \sim y$ precisely when $f(x) = f(y)$.

In light of this, given a set-indexed cocone $( f_j:U[(X_{j}, V_j,d_j)] \to X)_{j\in J}$ for each $j\in J$ we can apply the final continuity space construction by letting $X_j = Y$, $V=V_j$ and $d=d_j$ to obtain a collection of continuity spaces $(X,W_j,m_j)$ making their respective $f_j$'s $\epsilon$-$\delta$ continuous. It is important to keep in mind that, as illustrated above, for each $j\in J$ the continuity space $(X,W_j,m_j)$ is the $U$-final lift of $ f_j:U[(X_{j}, V_j,d_j)] \to X$. We start the second part of this proof: we construct the $U$-final lift, $(X,V,m)$, of $( f_j:U[(X_{j}, V_j,d_j)] \to X)_{j\in J}$ based on the collection $\{(X,W_j,m_j)\mid j\in J\}$. First we construct the coproduct $(Y,V,d)$ of the collection $\{(X_{j}, V_j,d_j)\mid j\in J\}$ and later apply the final continuity space to the obvious function $Y \to X$. Put 
\[
{\displaystyle N=\bigsqcup_{j\in J}(W_{j})_{\prec}}, \text{ and } Y=\bigsqcup_{j\in J}X_{j}
\]

\noindent
where $\bigsqcup$ denotes the disjoint union operator. It is not difficult to verify that $N$ is not a value distributive lattice (in fact, it is not even a lattice). That said, all information about the $W_i$'s is contained in $N$ and we construct $V$ in much the same way as with products: $V = \Omega(N)$. The distance function $d: Y^2\to V$ must be defined by parts: for a fixed $i\in J$ consider the function $\phi_i: W_i \to V$ so that 
for $a\in W_{i}$, 
\[{\displaystyle a \mapsto a_{\uparrow}:=\left[a^{\uparrow}\cup\bigsqcup_{i\not=j}(W_{j})_{\prec}\right]^{<\omega}}
\]
where $a^{\uparrow} = \{\epsilon \succ \bot_{i} \mid \epsilon \succ a\}$. In order to lighten some of the notational burden, for any $j\in J$ and a point $x\in i_j(X_j) \subseteq Y$, we let $x_j = i_j^{-1}(x)$ where $i_j: X_j \to Y$ denotes the obvious injection.
Lastly, define $d:Y^2 \rightarrow V$ for all $x,y\in Y$ as follows:
\[
d(x,y)=\begin{cases}
\phi_j\circ d_{j}(x_j,y_j) & x,y\in i_j(X_{j}),\\
\top_{V} & \mbox{otherwise.}
\end{cases}
\]\\
The following are easy to verify.\\

\noindent
{CLAIM:}  For $\{(X_j,V_j,d_j)\mid j\in J\}$ and $(Y,V,d)$ as above described we have:
\medskip
\begin{enumerate}
\item For any $j \in J$ the function $\phi_{j}$ is an order-embedding. 
\medskip
\item For any $j \in J$, all $\epsilon \succ 0_j$ and any pair $x,y\in i_j(X_j)$ then

$$d_j(x_j,y_j)\prec \epsilon \Leftrightarrow d(x,y) \prec \overline{\epsilon}=\{\{\epsilon\},\emptyset\}.$$
\medskip
\item For any $\top_V > \epsilon \succ \bot_V$ and $x,y \in X$: $d(x,y) \prec \epsilon$ if, and only if, for some $j\in J$ we have $x,y \in i_j(X_j)$ and $d_j(x_j,y_j) \prec \delta$, $\forall \delta \in \cup \epsilon \cap V_j$.
\end{enumerate}

\medskip

Next we show that the injections $i_j: (X_j,V_j,d_j) \to (Y,V,d)$ are $\epsilon$-$\delta$ continuous. Fix a $k\in J$, an $x\in i_k(X_k)$ and take any $p \in V_\prec$. If $\top_V= p$ then there is nothing to prove so we assume $\top_V > p$. Since $p \succ \bot_V$, there are only finitely many $q \in \cup p \cap (V_j)_\prec$ and thus $\delta:=\wedge q\succ \bot_k$. Moreover, $y_k \in B^{d_k}_\delta(x_k)$ implies $d_k(x_k,y_k) \prec p_n$ (for all $n$) and that $d(x,y) \prec p$. Hence, the injection $i_k$ is $\epsilon$-$\delta$ continuous. Next, take a continuity space $(Z,V',s)$ in conjunction with a collection $\epsilon$-$\delta$ continuous functions $g_j :X_j \to Z$ and let $h: Y \to Z$ be the canonical map $ x \mapsto f_j(x_j)$. Take $x \in Y$ and $\epsilon \succ \bot_{V'}$ where, without loss of generality, $x  = i_l(y)$ for some $l \in J$ and $y\in X_l$. Since $f_l$ is $\epsilon$-$\delta$ continuous there exists $\delta \succ \bot_l$ so that for all $z\in X_l$ we have $d_l(y,z) \prec \delta \Rightarrow s(f_l(y),f_l(z)) \prec \epsilon$. Letting $\overline{\delta} = \{\{\delta\},\emptyset\} \in V$ we have that 
\[
\displaystyle y \in B^{d}_{\overline{\delta}}(x) \Rightarrow y\in B^{d_l}_\delta(y)
\]

\noindent
 and, thus, $s(h(x),h(y)) = s (f_l(x_l),f_l(y_l)) \prec \epsilon$ making $h: Y \to Z$ $\epsilon$-$\delta$ continuous. As aforementioned, to complete this proof we need only to apply the final continuity space to $X$ based on $U\left[(Y,V,d)\right] \to X$ and denote it by $(X,W,m)$. It is only a matter of routine to check that $(X,W,m)$ is then the $U$-final lift of the cocone $( f_j:U[(X_{j}, V_j,d_j)] \to X)_{j\in J}$ as claimed.

\end{proof}

\subsubsection{Coproducts} Take an arbitrary collection of continuity spaces $\{(X_{j}, V_j,d_j)\} \mid j\in J\}$ and put 
\[
 V=\Omega\left[\bigsqcup_{j\in J}(W_{j})_{\prec}\right] \text{ and } X=\bigsqcup_{j\in J}X_{j}.
\]

\noindent
For a fixed $i\in J$ let $\phi_i: V_i \to V$ so that 
for $a\in V_{i}$, 
\[{\displaystyle a \mapsto a_{\uparrow}:=\left[a^{\uparrow}\cup\bigsqcup_{i\not=j}(V_{j})_{\prec}\right]^{<\omega}}
\]
where $a^{\uparrow} = \{\epsilon \succ \bot_{i} \mid \epsilon \succ a\}$. Lastly define $d:X^2 \rightarrow V$ for all $x,y\in X$ as follows:
\[
d(x,y)=\begin{cases}
\phi_j\circ d_{j}(x_j,y_j) & \text{ if }x,y\in i_j(X_{j}),\\
\top_{V} & \mbox{otherwise.}
\end{cases}
\]

\noindent
where $i_j : X_j \to X$ are the canonical injections, and $i_j(x_j)  = x$ and $i_j(y_j)  = y$.
The following follows directly from the proof of Theorem~\ref{thm:cotopological}.

\medskip
\begin{corollary} For a set-indexed collection of continuity spaces $\{(X_{j}, V_j,d_j)\mid j\in J\}$ we have $\bigamalg_{\bf P}  (X_j,V_j,d_j) = ((X,V,m), X_i \to X)$ with $X$, $V$ and $m$ as illustrated above.
\end{corollary}

\subsubsection{Coequalisers} Begin with a continuity space $(Y,W,m)$ and an equivalence relation ${\sim} \in Eq(Y)$. Put $X = Y/ {\sim}$, $M = (W_\prec)^Y$ and $V = \Omega[M]$ where $(W_\prec)^Y$ denotes the collection of all functions from $Y$ to $W_\prec$. The distance assignment $d:X^2 \to V$ is constructed based on epsilon balls. For $A \in\left[M\right]^{<\omega}$ and a pair $a,b \in X$ say that $A$ \emph{admits} $(a,b)$ if for all $h\in A$ there exists a finite sequence $x_1, \ldots, x_n$ in $Y$ so that: (a) $x_1 \in a$ and $x_n\in b $, (b) for $i$ odd, $x_{i+1}\in B^{m}_{h(x_i)}(x_i)$, and (c) for $i$ even $x_{i} \sim x_{i+1}$. Next, define $d: X^2 \to W$ as follows: for all $x,y\in X$ let

\[
d(x,y) = \{ A \in\left[M\right]^{<\omega}\mid A \text{ admits } (x,y)\}. 
\]\\

\begin{corollary} For any continuity space $((Y, W,m)$ and  ${\sim} \in Eq(Y)$ its coequaliser is $((X,V,d), Y \to X)$ with $X, V,$ and $d$ as above illustrated.
\end{corollary}

By cocontinuity of $\mathcal{O}$ the latter two corollaries illustrate continuity spaces that generate direct sum and quotient topologies.

\section{Conclusion} There are several concluding remarks that we must address. One deals with the $\epsilon$-$\delta$ continuous functions in $\bp$ and their definability purely in terms of topological properties. All of the information required for constructing colimits in {\bf Top} already exists among these $\epsilon$-$\delta$ continuous functions and it is unknown to the current authors if {\it nice} topological definitions for such exist.
Lastly, we note the topological (dis)-similarities between {\bf Top} and $\bp$. Although categorically similar, {\bf Top} and $\bp$ are topologically disparate. For instance, it is a simple matter to capture several separation axioms in terms of $\bm$-properties while in $\bp$ it is not even possible to claim that a continuity space is $T_0$ if any pair of points has a non-zero distance (see \cite{MR3334228}). Convergence in $\bp$, and thus compactness, also differs from compactness in {\bf Top}.

\section{Acknowledgments}
 We would like to thank Walter Tholen for his insightful suggestions, in particular the ones referring to Section 4. We would also like to extend our sincere gratitude to the referee for her/his many helpful comments.

\bibliographystyle{plain}
\bibliography{mybib}

\end{document}